\documentclass{amsart}
\usepackage[margin=1in]{geometry}
\usepackage{amsmath}
\usepackage{amssymb}
\usepackage{amsthm}
\usepackage{mathrsfs}
\usepackage{enumerate}
\usepackage{tikz}
\usetikzlibrary{decorations.pathreplacing}
\usepackage[indent]{parskip}
\newtheorem{thm}{Theorem}[section]
\newtheorem{lem}[thm]{Lemma}

\newtheorem{conj}[thm]{Conjecture}
\theoremstyle{definition}
\newtheorem{dfn}[thm]{Definition}

\newtheorem*{prb*}{Problem}

\theoremstyle{remark}
\newtheorem{rmk}[thm]{Remark}

\numberwithin{equation}{section}

\providecommand{\al}{\alpha}
\providecommand{\be}{\beta}

\providecommand{\e}{\varepsilon}

\providecommand{\emp}{\varnothing}

\providecommand{\N}{\mathbb{N}}

\providecommand{\R}{\mathbb{R}}

\newcommand{\conv}{\mathrm{conv}\thinspace}

\title{A non-linear Roth theorem for thick Cantor sets}

\author{Alex McDonald}
\address{Mathematics Department, Kennesaw State University, Marietta, GA}
\email{amcdon79@kennesaw.edu}

\author{Micah Nguyen}
\address{Mathematics Department, Kennesaw State University, Marietta, GA}
\email{mnguye94@students.kennesaw.edu}

\date{}

\begin{document}

\begin{abstract}
We prove that for any function $f$ satisfying certain mild conditions and any Cantor set $K$ with Newhouse thickness greater than $1$, there exists $x\in K$ and $t>0$ such that
\[
\{x-t,x,x+f(t)\}\subset K.
\]
This is an extension of previous work on the existence of three-term arithmetic progressions in Cantor sets to the non-linear setting.
\end{abstract}

\maketitle

\section{Introduction}
Many deep results in mathematics assert that if a set is sufficiently large, then it cannot avoid exhibiting some type of structure.  A classic example is Roth's Theorem \cite{Roth} in additive combinatorics, which says that sets of integers with positive upper density must contain three-term arithmetic progressions.  This means that if $E\subset \N$ satisfies
\begin{equation}
\label{positiveupperdensity}
    \limsup_{N\to\infty}\frac{|E\cap[1,N]|}{N}>0,
\end{equation}
then there exist $x\in E,t>0$ such that
\[
\{x-t,x,x+t\}\subset E.
\]
In analysis, the Falconer distance problem is an important variant of the ``largeness implies structure'' theme.  The problem asks how large the Hausdorff dimension of a set $E\subset \R^d$ must be to ensure that the distance set
\[
\Delta(E):=\{|x-y|:x,y\in E\}
\]
has positive Lebesgue measure.  This problem is open, with the best known results due to Guth, Iosevich, Ou, and Wang \cite{GIOW} in the plane and Du, Ou, Ren, Zhang \cite{DORZ} in higher dimensions.  More generally, it is an active area of current research to determine when sets of sufficiently large Hausdorff dimension are forced to contain certain types of finite point configurations.  While there are many positive results in this area, the analogue of Roth's Theorem is actually negative; for any dimension $d\in \N$, there is a set $E\subset \R^d$ with Hausdorff dimension $d$ which does not contain any three term arithmetic progression.  This is a special case of a well-known result of Maga \cite{Maga}, which states that for $d\in\N$ there is a set $E\subset \R^d$ of Hausdorff dimension $d$ which does not contain any set $\{x_1,x_2,x_3,x_4\}$ which has at least $3$ points distinct and satisfies $x_2-x_1=x_4-x_3$.  Any 3-AP $\{x-t,x,x+t\}$ with $t>0$ qualifies, as we may take $x_1=x-t,x_2=x_3=x,x_4=x+t$.

In recent years, there has been considerable interest in studying the existence of point configurations in Cantor sets under the assumption of large ``Newhouse thickness'', rather than large Hausdorff dimension.  Intuitively, suppose we construct a Cantor set the usual way, by starting with an interval and removing some open interval from the interior, and continue for countably many stages.  Unlike the middle-thirds Cantor set, we do not require the interval we remove to be a particular proportion of the lenght of the interval it is removed from, nor do we require it to be in the middle.  We also do not assume self-similarity; that is, we do not require that the ``same'' open interval is removed at each stage.  The thickness is defined by considering the ratio of what is left behind to what is removed at each stage, and taking the infimum of all such ratios.  The precise definitions are as follow.
\begin{dfn}
A \textbf{Cantor set} is a subset $K\subset\R^d$ which is non-empty, compact, perfect, and totally disconnected.
\end{dfn}
It is not hard to see that Cantor sets are precisely those which can be constructed in the process described in the previous paragraph.  For the remainder of the paper, we will restrict ourselves to the one-dimensional setting.
\begin{dfn}
\label{thickness}
If $K\subset\R$ is a Cantor set, then a \textbf{gap} of $K$ is a connected component of $\R\setminus K$.  If $u$ is the right endpoint of a gap $G$, the \textbf{bridge} at $u$ is the maximal interval $B=[u,v]$ with the property that for any gap $G'\subset B$, we have $|G'|\leq |G|$.  Bridges at left endpoints are defined similarly.  For any gap endpoint $u$ (right or left), define
\[
\tau(K,u)=\frac{|B|}{|G|}.
\]
The \textbf{Newhouse thickness} (or simply ``thickness'') of $K$ is
\[
\tau(K)=\inf_u \tau(K,u),
\]
the infimum being taken over all gap endpoints $u$.
\end{dfn}
Many point configuration problems for sets of large Hausdorff dimension are also interesting for sets of large thickness.  Simon and Taylor \cite{ST20} proved that if $E\subset \R^2$ is of the form $E=K_1\times K_2$ for Cantor sets $K_1,K_2$ satisfying the thickness condition $\tau(K_1)\tau(K_2)>1$, then the corresponding distance set has positive measure, solving an analogue of the Falconer problem in this context.  In fact, they prove the stronger result that for any $x\in\R^2$ the ``pinned'' distance set 
\[
\Delta_x(K_1\times K_2):=\{|x-y|:y\in K_1\times K_2\}
\]
contains a non-degenerate interval.  This was generalized by Taylor and the first author \cite{MT21, MT22} to show that products of thick Cantor sets contain certain types of tree structures.

In another direction, Yavicoli \cite{Y21} shows that any Cantor set of thickness at least $\tau$ contains a similar copy of any set with at most $N(\tau)$ points, where $N(\tau)$ is an explicit function.  This is obviously very flexible in terms of the patterns it allows one to find, but the function $N(\tau)$ grows quite slowly; to find a three point configuration using this theorem (i.e., to have $N(\tau)\geq 3$) one needs to take $\tau$ on the order of $10^9$.  More recently, Yavicoli \cite{YSurvey} gave a simple proof showing that a Cantor set $K$ satisfying $\tau(K)\geq 1$ must contain a 3-AP.  This was generalized by Sandberg-Clark and Taylor \cite{SCT}, who show that any Cantor set $K$ satisfying $\tau(K)\geq 1$and any $\lambda\in (0,1)$, $K$ must contain a configuration of the form $\{x,(1-\lambda)x+\lambda y,y\}$ for some $x\neq y$ (3-APs correspond to the case $\lambda=1/2$).  They also generalize this result to higher dimensions, using a notion of thickness introduced by Yavicoli in \cite{Y22}.

For a function $f$, we may ask what assumptions guarantee that a set contains configurations of the form $\{x-t,x,x+f(t)\}$ (3-APs correspond to the case $f(t)=t$).  If $f$ is non-linear, this problem is not invariant under scaling, so the results discussed above which pertain to sets containing ``similar copies'' of a given configuration do not apply.  In the discrete setting, Bergelson and Leibman \cite{BL96} proved that if $P_1,\dots,P_n$ are any polynomials with integer coefficients and constant term zero, and $E\subset \N$ satisfies (\ref{positiveupperdensity}), then there exists $t\neq 0$ such that
\begin{equation}
\label{BLequation}
    \{x,x+P_1(t),\dots,x+P_n(t)\}\subset E.
\end{equation}
In the continuous setting, Krause \cite{krause25} showed that if $\{P_1,\dots, P_n\}$ is a collection of polynomials satisfying a certain technical assumption, and $E\subset \R$ has $s$-dimensional Hausdorff content bounded away from zero for some $s$ sufficiently large, then there exists $t\neq 0$ such that (\ref{BLequation}) holds.

The purpose of this paper is to prove an analogous result for thick Cantor sets.  Our result does not require the function to be a polynomial.  However, we do impose bounds on the derivative which rule out some of the examples which are allowed in the aforementioned results.
\begin{thm}
\label{main}
Let $K\subset \R$ be a Cantor set with thickness $\tau>1$, and let $f:\R\to \R$ be a continuously differentiable function satisfying $f(0)=0$ and $\frac{1}{\tau}<f'(0)<\tau$.  Then, there exist $x\in \R,t> 0$ such that
\[
\{x-t,x,x+f(t)\}\subset K.
\]
\end{thm}
Note that the hypotheses of Theorem \ref{main} include the case $f(0)=0,f'(0)=1$.  Recall that 3-APs correspond to the function $f(t)=t$; hence, the theorem applies to any function which agrees with this one to the first order.  Before proving our result, we conclude the section with a short discussion of the extent to which our hypotheses are necessary.  If $f$ is continuous and $f(0)\neq 0$, then there exists $\e,\delta>0$ for which $f(t)>\e$ whenever $|t|<\delta$.  Suppose $K$ has thickness $\tau>1$, and $K$ has diameter less than both $\e$ and $\delta$.  If $\{x-t,x\}\subset K$, then $|t|<\delta$ and hence $f(t)>\e$, so $x+f(t)\notin K$.  This shows the assumption $f(0)=0$ is necessary.  It is less clear if the bounds on $f'$ can be improved.  Indeed, recall that in the linear setting it is known that $\tau>1$ implies the existance of a configuration $\{a,(1-\lambda)a+\lambda b,b\}$ for any $\lambda\in (0,1)$.  This is the same as $\{x-t,x,x+f(t)\}$ with $t=\lambda(b-a),x=a+\lambda(b-a)$, and $f(t)=\frac{1-\lambda}{\lambda}t$.  In other words, when $f$ is linear, it can have any non-zero slope and the result holds.  It seems reasonable to hope that the hypotheses of Theorem \ref{main} could be relaxed to $f(0)=0,f'(0)\neq 0$.  It would also be interesting to determine what, if anything, could be said in the case $f'(0)=0$.  In Section \ref{Necessity}, we will give a more thorough discussion of the obstacles which arise in trying to generalize our result.
\section{Basic properties of Newhouse thickness}
\label{properties}
The motivation for Definition \ref{thickness} is the following result (due to Newhouse \cite{Newhouse}; see also \cite{PTbook, YSurvey} for expositions) which gives a sufficient condition for sets to intersect.
\begin{thm}[Newhouse Gap Lemma]
\label{NGL}
Let $K_1,K_2\subset \R$ be Cantor sets satisfying $\tau(K_1)\tau(K_2)\geq 1$.  If neither $K_1$ nor $K_2$ is contained in a single gap of the other, then $K_1\cap K_2\neq\emp$.
\end{thm}
If $g(x)=ax+b$ is an affine map, it is easy to see $\tau(g(K))=\tau(K)$.  If $g:\R\to\R$ is continuous and monotone, then for any gap $G$ of $K$, $g(G)$ is a gap of $g(K)$.  However, if $B$ is a bridge of $K$, it is not necessarily true that $g(B)$ is a bridge of $g(K)$.  This causes thickness to behave unpredictably when a Cantor set is transformed by a non-linear map.  However, we have the following result, which says that thickness is ``locally almost preserved'' if we work close to a regular point of the map.
\begin{lem}[\cite{MT22}, Lemma 3.1]
\label{smooththickness}
Let $K\subset \R$ be a Cantor set, and let $I$ be an interval such that $K\cap I$ is a Cantor set with thickness $\tau(K\cap I)\geq \tau(K)$.  For any $0<\rho<1$, there exists $\e>0$ such that if $g$ is continuously differentiable and monotone on $I$ and
\[
\forall x,y\in I,\hspace{.1in}\left|\frac{|g'(x)|}{|g'(y)|}-1\right|<\e,
\]
then $\tau(g(K\cap I))>\rho\tau(K)$.
\end{lem}
The condition $\tau(K\cap I)\geq \tau(K)$ does not necessarily hold for all intervals $I$.  For example, if $K$ is the middle-thirds Cantor set and $I=[\frac{2}{9},\frac{7}{9}]$, then $\tau(K)=1$ but $\tau(K\cap I)=\frac{1}{3}$.  However, it is enough to ensure that the interval is a bridge, as the following lemma shows.
\begin{lem}
\label{bridgethickness}
Let $K\subset \R$, and let $B$ be some bridge of $K$.  Then,
\[
\tau(K\cap B)\geq \tau(K).
\]
\end{lem}

\begin{figure}
\centering
\begin{minipage}[b]{0.45\linewidth}
\begin{center}
\begin{tikzpicture}
%Axis
    \draw[ultra thick] (0,0)--(6,0);

%tics
   
    \draw (1,-0.1) -- (1,0.1);
    \draw (2,-0.1) -- (2,0.1);
    \draw (5,-0.1) -- (5,0.1);
    \draw (2.5,-0.1) -- (2.5,0.1);
    \draw (3,-0.1) -- (3,0.1);
    \draw (4,-0.1) -- (4,0.1);
    
%braces

    \draw [decorate,
    decoration = {brace}] (2,-.5) --  (1,-.5);
    \draw [decorate,
    decoration = {brace}] (5,-.5) --  (2,-.5);
    \draw [decorate,
    decoration = {brace}] (2.5,.5) --  (3,.5);
    \draw [decorate,
    decoration = {brace}] (3,.5) --  (4,.5);

%labels
   
    \node[below] at (1.5,-0.7) {$G$};
    \node[below] at (3.5,-0.7) {$B$};
    \node[above] at (2.75,0.7) {$G'$};
    \node[above] at (3.5,0.7) {$B'$};
    
\end{tikzpicture}
\end{center}
\caption{An ``interior'' bridge of $K\cap B$}
\label{CaseI}
\end{minipage}
\qquad
\begin{minipage}[b]{0.45\linewidth}
\begin{center}
\begin{tikzpicture}
%Axis
    \draw[ultra thick] (0,0)--(6,0);

%tics
   
    \draw (1,-0.1) -- (1,0.1);
    \draw (2,-0.1) -- (2,0.1);
    \draw (5,-0.1) -- (5,0.1);
    \draw (3.5,-0.1) -- (3.5,0.1);
    \draw (4,-0.1) -- (4,0.1);

%braces

    \draw [decorate,
    decoration = {brace}] (2,-.5) --  (1,-.5);
    \draw [decorate,
    decoration = {brace}] (5,-.5) --  (2,-.5);
    \draw [decorate,
    decoration = {brace}] (3.5,.5) --  (4,.5);
    \draw [decorate,
    decoration = {brace}] (4,.5) --  (5,.5);

%labels
   
    \node[below] at (1.5,-0.7) {$G$};
    \node[below] at (3.5,-0.7) {$B$};
    \node[above] at (3.75,0.7) {$G'$};
    \node[above] at (4.5,0.7) {$B'$};
    
\end{tikzpicture}
\end{center}
\caption{An ``endpoint'' bridge of $K\cap B$}
\label{CaseII}
\end{minipage}
\end{figure}

\begin{proof}
Consider some bridge $B$ of the cantor set $K$. Without loss of generality, let $B$ be the right bridge of some gap $G$. Let $G'$ be a bounded gap of $K\cap B$; this means $G'\subset B$, and hence $G'$ is also a gap of $K$.  We are going to show that if $B'$ is the left bridge of $G'$ with respect to the set $K\cap B$, then it is also the left bridge of $G'$ with respect to the set $K$.  From this, it follows that
\[
\frac{|B|}{|G|}\geq \tau(K).
\]
The same argument applies to the left bridges, hence $\tau(K\cap B)\geq \tau(K)$.  To prove the claim, we consider two cases: either $\max B'>\max K\cap B$ (Figure \ref{CaseI}), or $\max B'=\max K\cap B$ (Figure \ref{CaseII}).  In the first case, the gap $G''$ of $K\cap B$ which lies to the right of $B'$ is contained in $B$, and thus is a gap of $K$.  Similarly, all gaps of $K\cap B$ contained in $B'$ are contained in $B$, and thus also gaps of $K$.  It follows that $B'$ is the bridge of $G'$ with respect to $K$ also.  In the second case, the gap of $K\cap B$ lying to the right of $B'$ is the unbounded gap $(\max K\cap B,\infty)$.
Let $G''$ denote the gap of $K$ which lies to the right of $B'$.  If this is unbounded also (i.e., if $\max B'=\max K$), we are done.  Otherwise, we observe that since $B$ is a bridge of $K$ corresponding to $G$ and $G'\subset B$, we have $|G''|\geq |G|\geq |G'|$.  Any gap of $K$ contained in $B'$ must be smaller than $G'$, hence smaller than $G''$.  Therefore, in this case also, $B'$ is the bridge of $G'$ defined with respect to $K$.
\end{proof}
\begin{lem}
\label{subset}
Let $K\subset \R$ be a Cantor set.  For any $\delta>0$, there exists $K'\subset K$ such that $\tau(K')\geq \tau(K)$ and $|\conv K'|<\delta$.
\end{lem}
\begin{proof}
Let $G_0$ be some gap in $K$ with right endpoint $u$. Since $K$ is closed, $u$ is a limit point, and there exists some $v\in K$ such that $u<v<u+\min(|G_0|,\delta)$.  Then, since $K$ is totally disconnected, there exists a gap $|G|<|G_0|$ such that $w$ is the left endpoint of $G$ and $u<w<v$. Then the left bridge $B$ of $G$ exists and satisfies $B\subset [u,w]$, which implies $|B|<\delta$.  Finally, let $K'=K\cap B$. Then $|\conv K'|<\delta$, and by Lemma \ref{smooththickness}, $\tau(K')\geq\tau(K)$ as desired.
\end{proof}
\section{Proof of Theorem \ref{main}}
\subsection{The linear case}
Yavicoli \cite[Proposition 20]{YSurvey} proved that any Cantor set satisfying $\tau(K)\geq 1$ contains a three term arithmetic progression (hereafter, ``3-AP'').  In fact, Yavicoli's proof actually shows that uncountably many such progressions exist.  More specifically, Yavicoli shows that there is a non-degenerate interval $I$ intersecting $K$ such that any point in $K\cap I$ may be taken as the middle point of a progression.  We start by giving another proof of the existence of $3$-APs in thick Cantor sets.  Our argument only produces progressions where the middle point is  a gap endpoint of the Cantor set, hence it only produces countably many such progressions.  The upside is that the proof is simpler, allowing it to be adapted to the non-linear setting.
\begin{thm}
\label{linear}
Let $K\subset \R$ be a Cantor set of thickness $\tau\geq 1$.  Then, $K$ contains a $3$-AP.
\end{thm}
\begin{figure}
\label{abc_figure}
\begin{tikzpicture}
%Axis
    \draw[ultra thick] (1,0)--(13,0);
    \draw[ultra thick, red] (3,0)--(6,0);
    \draw[ultra thick, blue] (7,0)--(9,0);

%tics
    \draw (7,-0.1) -- (7,0.1);
    \draw (6,-0.1) -- (6,0.1);
    \draw (9,-0.1) -- (9,0.1);
    \draw (3,-0.1) -- (3,0.1);
    \draw (8,-0.1) -- (8,0.1);
    \draw (11,-0.1) -- (11,0.1);

%labels
    \node[below] at (7,-0.1) {$0$};
    \node[below] at (6,-0.1) {$-a$};
    \node[below] at (9,-0.1) {$c$};
    \node[below] at (3,-0.1) {$-b$};
    \node[below] at (8,-0.1) {$a$};
    \node[below] at (11,-0.1) {$b$};
\end{tikzpicture}
\caption{The convex hulls of $K_1$ (red) and $K_2$ (blue).}
\end{figure}
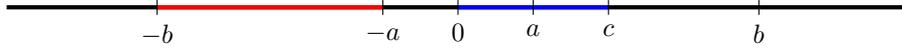
\begin{proof}
Let $G$ be the largest gap of $K$, and suppose without loss of generality that the left bridge of $G$ is at least as long as the right bridge.  After translating $K$ if necessary, we may suppose the largest gap of $K$ is $(-a,0)$ for some $a>0$.  Let $(-b,-a)$ and $(0,c)$ be the left and right bridges, respectively.  Let $K_1=K\cap (-b,-a)$ and $K_2=K\cap (0,c)$ (Figure \ref{abc_figure}).  If we can show $-K_1\cap K_2\neq\emp$, then $t\in -K_1\cap K_2$ would satisfy $-t\in K_1\subset K$ and $t\in K_2\subset K$, so $\{-t,0,t\}\subset K$ is a 3-AP contained in $K$.  Since $0\notin K_1$, we have $t\neq 0$, so this is a non-degenerate $3$-AP.  

To prove $K_1\cap K_2\neq\emp$, we use the Newhouse gap lemma.  By Lemma \ref{bridgethickness} we have $\tau(K_1)\geq 1$ and $\tau(K_2)\geq 1$, so $\tau(K_1)\tau(K_2)\geq 1$ also.  We already know $0<a$; if we can prove $a\leq c\leq b$, the gap lemma will apply (since $a,b\in -K_1$ would then have to lie in two different gaps of $K_2$, and $0,c\in K_2$ would lie in two different gaps of $-K_1$).  To prove these inequalities, first recall that we assumed the left bridge of $G$ is at least as long as the right bridge; in our notation, this is $b-a\geq c$, which implies $c\leq b$ immediately.  Next, since the right bridge has length $c$ and the gap has length $a$, we have $\tau\leq \frac{c}{a}$.  If $\tau\geq 1$, we therefore have $a\leq c$.
\end{proof}
\begin{rmk}
In the proof above, the 3-AP which is produced is of the form $\{-t,0,t\}$ because we assumed (without loss of generality) that the largest gap had right endpoint $0$, and that the left bridge was larger.  In general, the method produces a 3-AP $\{x,y,z\}$ where $y$ is the endpoint of the largest gap, on the side opposite the larger of the two bridges.
\end{rmk}
\subsection{The non-linear case}
\begin{lem}
\label{MVTlemmav3}
Let $a,b,c>0$ be numbers satisfying
\begin{equation}
\label{parameterconstraints}
a<c\leq b-a,
\end{equation}
and let $g$ be a function differentiable on $(0,c)$ which satisfies
\begin{equation}
\label{dvconstraint}
\frac{a}{c}<g'<\frac{c}{a}.
\end{equation}
Let $K$ be a Cantor setwith largest gap $(-a,0)$ and left and right bridges $[-b,-a]$ and $[0,c]$, respectively.  Finally, define
\[
K_L=K\cap [-b,-a] \hspace{.25in}\textup{and}\hspace{.25in} K_R=K\cap [0,c].
\]
then, the sets $-K_L$ and $g(K_R)$ are interleaved; that is, neither is contained in a single gap of the other.
\end{lem}
\begin{proof}
    The equation (\ref{dvconstraint}) together with the assumption $g(0)=0$ implies that $g$ is stricly increasing, so 
    \[
    \conv(g(K_R))=[0,g(c)].
    \]
    By the Mean Value Theorem, there exists $\xi\in (0,c)$ such that $g(c) = g'(\xi)c$. The lower bound in (\ref{dvconstraint}) then implies $g(c)>a$.  If $g(c)<b$ also, this completes the proof, as the intervals $\conv(-K_L)=[a,b]$ and $\conv(g(K_R))=[0,g(c)]$ have alternating endpoints.  If instead $g(c)>b$, then clearly $g(K_R)$ cannot lie in a single gap of $-K_L$.  It remains to show $-K_L$ does not lie in a single gap of $g(K_R)$ either.  Since $-K_L\subset \conv(g(K_R))$ in this case, we cannot have $-K_L$ in an unbounded gap of $g(K_R)$.  To show it does not lie in a single bounded gap, first note that any bounded gap of $g(K_R)$ is of the form $g(G)$, where $G$ is a bounded gap of $K_R$.  Because $\conv(K_R)$ is the left bridge (with respect to $K$) of the gap $(-a,0)$, we have $|G|\leq a$.  By the mean value theorem and the upper bounds in (\ref{dvconstraint}) and (\ref{parameterconstraints}), we have
    \[
    |g(G)|< \frac{c}{a}|G|\leq c\leq b-a=|\conv(-K_L)|.
    \]
   Therefore, $-K_L$ does not lie in $g(G)$.
\end{proof}

We are now ready to prove Theorem \ref{main}.
\begin{proof}[Proof of Theorem \ref{main}]

We want to choose $K'\subset K$ such that $K'$ satisfies Lemma \ref{MVTlemmav3}. To do this, first choose $0<\rho<1$ so that $\rho\tau\geq 1$, and let $\e>0$ be as given by Lemma \ref{smooththickness}.  Finally, let $\delta>0$ be such that the following inequalities hold for all $|x|<\delta$, where $g$ may denote either $f$ or $f^{-1}$:
\begin{align}
    \label{first} |g'(x)-g'(0)|<\frac{\e}{2\tau}&, \\[.1in]
     \label{second} \frac{1}{\tau}<g'(x)<\tau &.
 \end{align}
This is possible by continuity of $f$ and $f'$, and by the bounds on $f'(0)$ (and hence $(f^{-1})'(0)$) assumed in the statement of the theorem.  In particular, note that (\ref{first}) and (\ref{second}) together imply that the hypothesis of Lemma \ref{smooththickness} is met for both $g=f$ and $g=f^{-1}$, as
\[
    |g'(x)-g'(y)|\leq |g'(x)-g'(0)|+|g'(y)-g'(0)|
    <\frac{\e}{\tau} \\[.1in]
    <\e|g'(y)|,
\]
hence
\[
\left|\frac{g'(x)}{g'(y)}-1\right|<\e.
\]
By Lemma \ref{subset} we may choose $K'\subset K$ with $\tau(K')\geq \tau$ and $|\text{conv}\: K'|<\delta$.  Let $G$ be the largest gap of $K'$, and let $B_L,B_R$ be the left and right bridges of $G$.  There are two cases to consider, depending which bridge is larger.

First, suppose $|B_L|\geq |B_R|$.  After translation, we may asssume $G=(-a,0)$ for some $a>0$. Let $b,c>0$ be such that $B_L=[-b,-a]$ and $B_R=[0,c]$.  We claim that Lemma \ref{MVTlemmav3} applies with this choice of parameters.  To verify the hypotheses of Lemma \ref{MVTlemmav3}, first note that $|B_L|\geq |B_R|$ implies $b-a\geq c$, giving the upper bound of (\ref{parameterconstraints}).  Since $c=|B_R|$ and $a=|G|$, we have $\frac{c}{a}\geq \tau>1$ and $\frac{a}{c}\leq\frac{1}{\tau}<1$, giving both the lower bound of (\ref{parameterconstraints}) as well as both bounds of (\ref{dvconstraint}).  If $K_L=K'\cap [-b,-a]$ and $K_R=K'\cap [0,c]$, we conclude from Lemma \ref{MVTlemmav3} that $-K_L$ and $f^{-1}(K_R)$ are interleaved.  By Lemma \ref{bridgethickness}, 
\[
\tau(K_L)\geq\tau(K')>1.
\]
By Lemmas \ref{smooththickness} and \ref{bridgethickness}, 
\[
\tau(f^{-1}(K_R))>\rho\tau(K_R)\geq\rho\tau\geq 1.
\]
It follows from the Newhouse gap lemma (Theorem \ref{NGL}) that there exists some $t\in-K_1\cap f^{-1}(K_2)$.  This implies $-t\in K_1\subset K$ and $f(t)\in K_2\subset K$, so $\{-t,0,f(t)\}\subset K'$ as desired.

Finally, we consider the case $|B_R|\geq |B_L|$.  The set $-K'$ then fits the case considered above.  Running the same argument with $f$ in place of $f^{-1}$, we obtain a configuration $\{-s,0,f^{-1}(s)\}\subset -K'$ for some $s>0$.  Letting $t=f^{-1}(s)$ we have $\{-f(t),0,t\}\subset -K'$, or $\{-t,0,f(t)\}\subset K'$ as desired.

\end{proof}

\section{Necessity of derivative assumptions}
\label{Necessity}
\subsection{Goals for future work: The non-degenerate case}
Our proof of Theorem \ref{main} was based on perturbing our proof of Theorem \ref{linear}.  The bounds on the derivative of $f$ seem likely to be an artifact of that approach.  We conjecture the following.
\begin{conj}
    Let $K\subset \R$ be a Cantor set satisfying $\tau(K)>1$, and let $f$ be a continuously differentiable function satisfying $f(0)=0,f'(0)> 0$.  Then, there exist $x\in\R,t>0$ such that
    \[
    \{x-t,x,x+f(t)\}\subset K.
    \]
\end{conj}
In \cite{SCT}, Sandberg-Clark and Taylor prove exactly this conjecture in the linear case; that is, for functions $f(t)=mt$ where $m>0$.  Their proof involves the observation that $(a,b,c)=(x-t,x,x+mt)$ for some $x,t$ if and only if
\[
b=\frac{m}{m+1}a+\frac{1}{m+1}c.
\]
This means that finding such a progression amounts to proving
\[
K\cap\left(\frac{m}{m+1}K+\frac{1}{m+1}K\right)\neq\emp.
\]
Using the Newhouse Gap Lemma, one can show that the sumset of thick Cantor sets is an interval, and this interval can be described explicitly in terms of the convex hulls of the sets being summed.  With an explicit description, it is not difficult to show $K$ intersects this interval, which completes the proof.

To extend this idea to the non-linear setting, the analogous starting observation is that $(a,b,c)=(x-t,x,x+f(t))$ for some $x,t$ if and only if
\[
c=b+f(b-a).
\]
Defining a new function $F(x,y)=y+f(y-x)$, in order to find the desired configuration, it is enough to show
\[
K\cap F(K,K)\neq\emp.
\]
Thus, the non-linear function $F$ plays the role that the sumset played in the linear setting.  The problem of describing the image of a thick Cantor set under such a function is one which should be explored in further work.
\subsection{A critical example in the degenerate case}
While the exact bounds on $f'$ are likely an artifact of the proof, the assumption that $f'(0)\neq 0$ is likely more fundamental.  Recall that our proofs in both the linear and non-linear setting (Theorems \ref{linear} and \ref{main}), as well as the proof of the linear result in \cite{SCT}, show that one may find a configuration $\{x-t,x,x+f(t)\}$ where $x$ is one of the two endpoints of the largest gap of $K$.  We conclude by giving a simple example of a compact set $K$ which avoids these patterns for $f(t)=t^2$.  This does not show that the assumption $f'(0)\neq 0$ is strictly necessary for the conclusion of Theorem \ref{main} to hold, since the configuration $\{x-t,x,x+t^2\}$ may still occur for some $x\in K$ other than one of those two points.  However, it does show that the proof technique cannot be extended to this case.  We also note that while our construction below is a finite union of intervals rather than a Cantor set, it can easily be reduced further to a Cantor set with the same thickness if desired.

\begin{thm}
\label{squareexample}
    There exists an absolute constant $\tau_0>1$ with the following property.  For any $\tau\in (1,\tau_0)$, there exists a compact set $K$ of thickness $\tau$ which does not contain a configuration of the form $\{x-t,x,x+t^2\}$ where $t>0$ and $x$ is one of the two endpoints of the largest gap of $K$.
\end{thm}
\begin{figure}
\begin{tikzpicture}
%segments
    \draw[ultra thick] (-8,0)--(-6,0);
    \draw[ultra thick] (-4,0)--(-2,0);
    \draw[ultra thick] (0,0)--(.5,0);
    \draw[ultra thick] (.95,0)--(1.45,0);
    \draw[ultra thick] (2,0)--(6,0);

%labels
    \node[below] at (-7,0) {$I_1$};
    \node[below] at (-5,0) {$G_1$};
    \node[below] at (-3,0) {$I_2$};
    \node[below] at (-1,0) {$G_2$};
    \node[below] at (.25,0) {$I_3$};
    \node[below] at (.725,0) {$G_3$};
    \node[below] at (1.225,0) {$I_4$};
    \node[below] at (1.725,0) {$G_4$};
    \node[below] at (4,0) {$I_5$};

    \node[above] at (-7,.5) {$B_1^L$};
    \node[above] at (-3,.5) {$B_1^R$};
    \node[above] at (.25,.5) {$B_3^L$};
    \node[above] at (1.225,.5) {$B_3^R$};

    \node[above] at (.725,1.2) {$B_4^L$};
    \node[above] at (4,1.2) {$B_4^R$};

    \node[above] at (-5,1.9) {$B_2^L$};
    \node[above] at (3,1.9) {$B_2^R$};

%Braces
    \draw [decorate,
    decoration = {brace}] (-8,.5) --  (-6,.5);
    \draw [decorate,
    decoration = {brace}] (-4,.5) --  (-2,.5);
    \draw [decorate,
    decoration = {brace}] (0,.5) --  (.5,.5);
    \draw [decorate,
    decoration = {brace}] (.95,.5) --  (1.45,.5);

    \draw [decorate,
    decoration = {brace}] (0,1.2) --  (1.45,1.2);
    \draw [decorate,
    decoration = {brace}] (2,1.2) --  (6,1.2);

    \draw [decorate,
    decoration = {brace}] (-8,1.9) --  (-2,1.9);
    \draw [decorate,
    decoration = {brace}] (0,1.9) --  (6,1.9);
\end{tikzpicture}
\caption{The set $K$ in Theorem \ref{squareexample} with bridges labeled.}
\label{square_figure}
\end{figure}
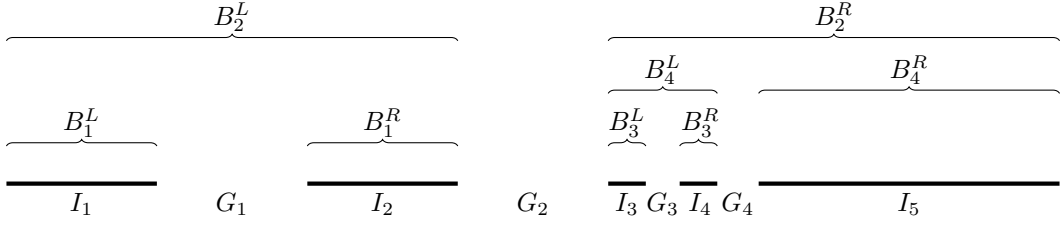
\begin{proof}
    Let $\tau>1,\e>0,0<c<1$ be parameters to be determined later.  Also let $\alpha=\frac{1}{2\tau+1}$ and $\beta=\frac{\tau}{2\tau +1}$; these figures represent the proportion of an interval which must be removed from the center to obtain a set of thickness $\tau$, and the proportion left on each side, respectively.  For future reference, we note the relations $2\beta+\alpha=1$ and $\beta/\alpha=\tau$.  Let
    \[
    K=I_1\cup I_2\cup I_3\cup I_4\cup I_5,
    \]
    where
    \begin{align*}
    I_1&= [-(1+\tau)\e,-(1+\beta\tau+\alpha\tau)\e] \\[.1in]
    I_2&= [-(1+\beta\tau)\e,-\e] \\[.1in]
    I_3&= [0,c\e^2] \\[.1in]
    I_4&= [(1+\beta\tau)^2c^{-1}\e^2,(1+\beta\tau+\alpha\tau)^2c\e^2] \\[.1in]
    I_5&= [(1+\tau)^2c^{-1}\e^2,\tau\e].
    \end{align*}
    The set $K$ is shown in Figure \ref{square_figure}.  For $i=1,2,3,4$, let $G_i$ be the gap between $B_i$ and $B_{i+1}$.  The largest gap of $K$ is $G_2=(-\e,0)$, so our claim is that $K$ does not contain any configuration of the form $\{x-t,x,x+t^2\}$ for $t>0$ and $x\in \{-\e,0\}$.  By construction, $t\in I_1$ implies $t^2\in G_4$ and $t\in I_2$ implies $ t^2\in G_3$.  Therefore, $K$ does not contain any configuration of the form $\{-t,0,t^2\}$ for $t>0$.  Also, if $\e$ is sufficiently small, then $t^2<\e$ for any $t\in K$, so $K$ does not contain configurations of the form $\{-\e-t,-\e,-\e+t^2\}$ for $t>0$.  It remains to show $\tau(K)=\tau$ when $\tau$ is sufficiently close to $1$.  The following table lists the intervals left to right, along with their lengths (the notation $\approx$ denotes the limit as $c\to 1$; the rows in which this symbol does not appear are independent of $c$).  The last column gives the lenghts in the case $\tau=1$ (note that in this case, $\al=\be=\frac{1}{3}$).
    \begin{center}
    \begin{tabular}{c|c|c}
    Interval & Length ($c\to 1$) & Length ($c\to 1,\tau=1$) \\
    \hline
    $I_1$ & $\beta\tau\e$ & $\e/3$ \\[.1in]
    $G_1$ & $\alpha\tau\e$ & $\e/3$ \\[.1in]
    $I_2$ & $\beta\tau\e$ & $\e/3$ \\[.1in]
    $G_2$ & $\e$ & $\e$ \\[.1in]
    $I_3$ & $\approx\e^2$ & $\approx \e^2$ \\[.1in]
    $G_3$ & $\approx (\beta^2\tau^2+2\beta\tau)\e^2$ & $\approx 7\e^2/9$ \\[.1in]
    $I_4$ & $\approx \al\tau(2+\tau)\e^2$ & $\approx\e^2$ \\[.1in]
    $G_4$ & $\approx\be\tau(2+3\be+2\al)\e^2$ & $\approx 11\e^2/9$ \\[.1in]
    $I_5$ & $\approx \tau\e-(1+\tau)^2\e^2$ & $\approx \e-4\e^2$
    \end{tabular}
    \end{center}
Let $B_i^L,B_i^R$ be the left and right bridge, respectively, of the gap $G_i$.  By continuity and the calculations in the above table, if $\tau$ and $c$ are sufficiently close to $1$, the bridges are as shown in Figure \ref{square_figure}.
For each $i$ and each choice of $L$ or $R$, when $\tau$ and $c$ are sufficiently close to $1$, we have
\[
\frac{|B_i^{L/R}|}{|G_i|}\geq \tau,
\]
with equality holding in several cases.  However, in all cases where the length depends on $c$, the inequality is strict.  Therefore, for any $\tau>1$ sufficiently close to $1$, there is a choice of $c$ which gives $\tau(K)=\tau$.  The threshold for $\tau$ depends only on the minimum of a finite number of zeros of polynomials, so there is an absolute constant $\tau_0$ with the desired property.
\end{proof}

\bibliographystyle{plain}
\bibliography{references}

\end{document}